\documentclass[11pt,letter]{article}

\usepackage[utf8]{inputenc}
\usepackage{amsmath}
\usepackage{amssymb}
\usepackage{amsfonts}
\usepackage{amscd}
\usepackage{latexsym}
\usepackage{wasysym}
\usepackage{graphicx}
\usepackage{xcolor}
\usepackage{hyperref}
\usepackage[all,cmtip]{xy}
\usepackage{bbold}

\usepackage{amssymb}
\usepackage{amsthm}
\usepackage{enumitem}
\usepackage{mathtools,bbm}
\usepackage{mathtools,bbm}
\usepackage{xpatch}
\usepackage{version,xspace}

\usepackage[margin=1in]{geometry}
\usepackage{hyperref}%
\hypersetup{colorlinks=true, linkcolor=blue, breaklinks=true, urlcolor=blue}
\usepackage{url,doi}

\usepackage[utf8]{inputenc}
\usepackage{amsfonts}
\usepackage{amsthm}
\usepackage{enumitem}
\usepackage{mathtools,bbm}
\usepackage{mathtools,bbm}
\usepackage{xpatch}

\usepackage{lipsum}
\usepackage{mathtools}
\usepackage{cuted}
\usepackage{mathtools,bbm}

\usepackage[caption=false,font=footnotesize]{subfig}
\usepackage[font=small]{caption}
\graphicspath{{./images/}{./fig/}}

\usepackage{enumitem}
\setlist{nosep}

\newcommand{\mcL}{\mathcal{L}}
\newcommand{\mcX}{\mathcal{X}}

\newcommand{\BX}{B(\mcX)}

\newcommand{\Y}{\mathcal{Y}}

\newcommand{\inpr}[2]{\langle#1,#2\rangle}
\newcommand{\inprX}[2]{\langle#1,#2\rangle_X}
\newcommand{\inprY}[2]{\langle#1,#2\rangle_Y}
\newcommand{\semin}[2]{\langle\!\langle#1,#2\rangle\!\rangle}

\newcommand{\inprPRS}[2]{\langle#1,#2\rangle_{\Pi_{\mathcal{S}},P^{1/2}}}

\newcommand{\T}{\mathcal{T}}
\newcommand{\jac}[1]{D\mkern-2.5mu{#1}}

\newcommand{\R}{\mathbb{R}}

\renewcommand{\ker}{\mathrm{Ker}}
\newcommand{\im}{\mathrm{Im}}

\newcommand{\op}{\operatorname{op}}

\newcommand{\until}[1]{\{1,\dots, #1\}}

\newtheorem{theorem}{Theorem}[section]

\newtheorem{lemma}{Lemma}[section]

\newtheorem{remark}[theorem]{Remark}

{
      \theoremstyle{plain}

}
\newtheorem{definition}{Definition}[section]

\DeclareSymbolFont{bbold}{U}{bbold}{m}{n}
\DeclareSymbolFontAlphabet{\mathbbold}{bbold}
\newcommand{\vect}[1]{\mathbbold{#1}}

\usepackage{algpseudocode,algorithm,algorithmicx}

\usepackage{hyperref}%
\hypersetup{colorlinks=true, linkcolor=blue, breaklinks=true, urlcolor=blue}

\newcommand{\real}{\ensuremath{\mathbb{R}}}

\newcommand\norm[1]{\left\lVert#1\right\rVert}
\renewcommand{\norm}[1]{\|#1\|}

\newcommand{\snorm}[1]{{\left\vert\kern-0.25ex\left\vert\kern-0.25ex\left\vert #1 
    \right\vert\kern-0.25ex\right\vert\kern-0.25ex\right\vert}}
\newcommand\oprocendsymbol{\hbox{$\square$}}
\newcommand\oprocend{\relax\ifmmode\else\unskip\hfill\fi\oprocendsymbol}

\DeclareSymbolFont{bbold}{U}{bbold}{m}{n}
\DeclareSymbolFontAlphabet{\mathbbold}{bbold}

\usepackage[prependcaption,colorinlistoftodos]{todonotes}

\title{Contraction Theory for Dynamical Systems \\ on Hilbert Spaces}

\author{Pedro Cisneros-Velarde, Saber Jafarpour, Francesco Bullo
  \thanks{Pedro Cisneros-Velarde (pacisne@gmail.com), Saber Jafarpour and Francesco Bullo (\{saber,bullo\}@ucsb.edu) are with the Center for
    Control, Dynamical Systems and Computation, University of California,
    Santa Barbara.}
    }    
\date{}

\begin{document}
\maketitle

\begin{abstract}
  Contraction theory for dynamical systems on Euclidean spaces is
  well-established.
  For contractive (resp. semi-contractive) systems, the distance
  (resp. semi-distance) between any two trajectories decreases
  exponentially fast.
  For partially contractive systems, each trajectory converges
  exponentially fast to an invariant subspace.

  In this note, we develop contraction theory on Hilbert spaces.
  First, we provide a novel integral
  condition for contractivity, and for time-invariant systems, we establish the existence of a unique
  globally exponentially stable equilibrium.
  Second, we introduce the notions of partial and semi-contraction and we
  provide various sufficient conditions for time-varying and time-invariant
  systems.
  Finally, we apply the theory on a classic reaction-diffusion system.
\end{abstract}

  %\sabermargin{This is not a good definition (or intuitive definition) of
 % semi-contraction. I recommend this one: a semi-contractive system is the
 % one for which the distance in certain directions between every two
 % trajectories
%  decreases exponentially}, 


\section{Introduction}

\paragraph*{Problem statement and motivation}
Contraction theory establishes the exponential incremental stability of
ordinary differential equations. Its mature development can be traced back
to the work by Coppel~\cite{WAC:1965}, where linear systems were studied,
and to the textbook treatment by Vidyasagar~\cite{MV:02}. Later, a
reformulation was proposed in the seminal work by
Slotine~\cite{WL-JJES:98}. We refer to~\cite{ZA-EDS:14b} for an
introduction and a survey of applications on contraction theory, and
to~\cite{JWSP-FB:12za} for extensions to Riemannian manifolds.
Generalizations of the classical contraction theory have been proposed in
the literature. The notion of partial contraction, first introduced
in~\cite{QCP-JJS:07}, studies the exponential convergence of trajectories
to invariant
subspaces~\cite{QCP-JJS:07,MdB-DF-GR-FS:16}. Recently,~\cite{SJ-PCV-FB:19q}
introduces the concept of semi-contraction, which establishes the
exponential incremental semi-stability of trajectories. Contraction
theory has also been used for control design~\cite{IRM-JJES:17,HSS-MR-IRM:19},  
and extended to non-differentiable and discontinuous vector fields~\cite{WL-MdB:16,DF-SJH-MdB:16,SAB-TL-SDC:18} and dynamics on Finsler manifolds~\cite{FF-RS:14}.

To the best of our knowledge, a general contraction theory on Hilbert and
Banach spaces is missing. The importance of working with systems defined on
such general space is, for example, the wide scope of possible applications
of systems based on partial differential equations, delayed differential
equations, functional differential equations, and integro-differential
equations (e.g., see~\cite[Chapter~9]{ANM-LH-DL:08}). The purpose of this
note is to concisely present such a theory, with the hope that it will be
relevant in both theoretical and applied work.  We also provide an
application example to illustrate the theory.  This work builds a bridge
between the abstract theory of differential equations developed in
mathematics~\cite{JLD-MGK:02}\cite{GEL-VL:72} and the widely-established
contraction theory in the field of systems and control.

\paragraph*{Literature review}
To the best of our knowledge, a first approach to contraction theory on
general Banach spaces can be traced back to the 1972 book by Ladas \&
Lakshmikantham~\cite{GEL-VL:72}, in its Lemma~5.4.1 and~5.4.2.  However,
these results do not parallel much of the richer development of contraction
theory in Euclidean spaces (see our Contributions below).  Interestingly,
the results in~\cite{GEL-VL:72} seem to be unknown in the literature on
contraction theory, which developed decades later. Applications of
contraction theory have been proposed to specific classes of partial
differential equations~\cite{ZA-EDS:14b,ZA-YS-MA-EDS:14,ZA-ES:13} and more
recently to functional differential equations~\cite{PHAN-HT:18}.  Besides
these notable exceptions, the study of contraction theory on infinite
dimensional systems has not received the same development as the Euclidean
case, e.g., no concept of semi- or partially contractive systems on Hilbert
spaces exists in the literature either.

In the controls community, the recent
works~\cite{ATJ-VA-VDSM-CX:19,MK:19} have considered dynamical
systems on Banach and Hilbert spaces and their applications to PDEs. Other
recent interests in dynamical systems on these abstract spaces include
controller design~\cite{AP-JC-LR:18}, event-triggered
control~\cite{MW-HS:20}, observability studies~\cite{DG-CS-MT:20}, optimal
control~\cite{ET-CZ-EZ:18}, and stability
characterizations~\cite{AM-FW:18}.

%; and the study of differential inclusions~\cite{ZL-XL-DM:15}, stochastic
%; differential equations~\cite{VB-MR-DZ:20} and switched
%; systems~\cite{VS-CT:19}.}

\paragraph*{Contributions}
First, we review two little-known results from~\cite{GEL-VL:72} that
establish a generalization of Coppel's inequality to Banach spaces and that
provide some sufficient conditions for contraction using \emph{operator
  measures} when the vector fields are continuously differentiable.  Then,
we prove that every time-invariant contractive system has a unique globally
exponentially stable equilibrium point. We also provide a sufficient
condition using operator measures for when the norm of a time-invariant
system has its vector field exponentially decreasing on trajectories of the
system.  In the case of systems on Hilbert spaces, we introduce a simpler
sufficient condition for contraction without the differentiability
requirement on the vector field: the \emph{integral contractivity
  condition}. Moreover, under the differentiability requirement, we prove:
(i) that the condition using operator measures presented
in~\cite{GEL-VL:72} can be relaxed and still imply contraction (in
particular, it is no longer needed to check a bound for Fr\'{e}chet
  derivative of the vector field); (ii) the integral contractivity
condition is implied by the one using operator measures.

Second, associated with a surjective linear operator $\T$, we introduce the
concepts of \emph{$\T$-seminorms} and \emph{$\T$-operator semi-measures}
which can be considered as generalization of recently introduced concepts
in the study of the classical Euclidean setting~\cite{SJ-PCV-FB:19q}.
Then, we introduce the concepts of partial and semi-contraction for systems
on Hilbert spaces. Using the concepts of seminorms and semi-measures, we
provide sufficient conditions for partial contraction and semi-contraction.
We present a series of novel results.  Firstly, we introduce the
\emph{integral partial contractivity condition}, a sufficient condition for
partial contraction.  Secondly, we introduce the \emph{integral
  semi-contractivity condition}, a sufficient condition for
semi-contraction. For continuous differentiable vector fields, we prove
this condition is implied by another sufficient condition for
semi-contraction using operator semi-measures. When there exists an
invariant subspace for the system, our conditions for semi-contraction
imply partial contraction. We remark that, to the best of our knowledge,
our characterization of partial and semi-contraction using integral
conditions are new even in the classic Euclidean setting (with the usual
inner product); e.g. as studied in the work~\cite{QCP-JJS:07} and in our
previous work~\cite{SJ-PCV-FB:19q}.

Finally, we present an example of a reaction-diffusion system and use
partial contraction to prove the same result as~\cite{MA:11}; moreover, we
establish semi-contraction when the reaction term is linear in the state
variable.
%
%\fbtodo{let's talk about this paragraph, I find the first sentence vacuous and
%the second sentence redundant with the above list of contributions}
%\pc{We remark that the challenge of generalizing the concepts from
%  classical contraction theory to the infinite dimensional case lies upon
%  the translation of its different associated concepts and main results to
%  a more general machinery drawn from functional analysis. 
%%  For example, the
%%  operator theoretic definition of partial and semicontraction based on
%%  inner products and seminorms in Section~\ref{sec:part-contr-HS} is novel
%%  and allows us to extend the classical results to infinite-dimensional
%%  spaces.
%  }

\paragraph*{Paper organization} 
Section~\ref{sec:prelim} has preliminaries and
notation. Sections~\ref{sec:contr_main} and~\ref{sec:part-contr-HS} contain
the main results on Banach and Hilbert spaces. Section~\ref{sec:examples}
presents the application example and Section~\ref{sec:concl} is the
conclusion.

\section{Preliminaries and notation}
\label{sec:prelim}

\subsection{Notation, definitions and useful results}

%\pc{Consider a vector space $\mcX$ over the field of real numbers.} 
%A \emph{Banach space} is a complete normed vector space $(\mcX,\norm{\cdot})$, where 
%%$\mcX$ is a vector space and 
%$\norm{\cdot}$ \pc{is} a norm over $\mcX$.  A \emph{Hilbert space}
%is a pair $(\mcX,\inpr{\cdot}{\cdot})$, where %$\mcX$ is a vector space and
%$\inpr{\cdot}{\cdot}$ is an inner product over $\mcX$, such that its induced norm $\norm{\cdot}:=\sqrt{\inpr{\cdot}{\cdot}}$ makes the space a Banach space. 
%%In what follows we assume $\mcX$ is a vector space over the field of real numbers.
A \emph{Banach space} is a complete normed vector space $(\mcX,\norm{\cdot})$, where $\mcX$ is a
vector space and $\norm{\cdot}$ a norm over $\mcX$.  A \emph{Hilbert space}
is a pair $(\mcX,\inpr{\cdot}{\cdot})$, where $\mcX$ is a vector space and
$\inpr{\cdot}{\cdot}$ is an inner product over $\mcX$, such that its induced norm $\norm{\cdot}:=\sqrt{\inpr{\cdot}{\cdot}}$ makes the space a Banach space. In what follows we
assume $\mcX$ is a vector space over the field of real numbers.

Let $B(\mcX)$ be the space of bounded linear operators with domain and codomain
$\mcX$. Let $0$ be the null element of $\mcX$, or the number zero,
depending on the context.  Let $I$ be the identity operator.  Given 
%an operator 
$A\in\BX$, $\norm{A}_{\op}=\sup_{\substack{x\neq
    0\\x\in\mcX}}\frac{\norm{Ax}}{\norm{x}}$ is its associated operator
norm.  Given an open set $\Omega\subseteq\mcX$, we say a function
$H:\mcX\to\mcX$ is continuously Fr\'{e}chet differentiable in $\Omega$ when
$H$ is Fr\'{e}chet differentiable at each $x_o\in\Omega$ (i.e., $DH(x_o)\in\BX$ such that $\lim_{\norm{e}\to 0}\frac{\norm{H(x_0+e)-H(x_0)-DH(x_0)e}}{\norm{e}}=0$
exists~\cite{RA-JEM-TSR:88}) and $DH:\Omega\to\BX$ is continuous.
Finally, we say a subspace $\mathcal{V}\subset\mcX$ is invariant for $A\in\mcX$ if for any $x\in\mathcal{V}$ then $Ax\in\mathcal{V}$.

Let $I_n$ be the $n\times n$ identity matrix and $\vect{0}_n\in\real^n$ be
the all-zeros column vector with $n$ entries.

The concept of matrix measures or logarithmic norm, e.g., see~\cite{ZA-EDS:14b}, can be generalized as the following; e.g., see~\cite[Definition~5.4.2]{GEL-VL:72},

\begin{definition}[Operator measure]
\label{def:oper-ms}
Let $A\in\BX$ and define the \emph{operator measure} of $A$ as:
$$
\mu(A)=\lim_{h\to 0^+}\frac{\norm{I+hA}_{\op}-1}{h}.
$$
\end{definition}

\subsection{Dynamical systems on Banach spaces}

Given the Banach space $(\mcX,\norm{\cdot})$ and the vector field
$F:\R\times{\mcX}\to \mcX$, consider the differential equation:
\begin{equation}
\label{eq:dyn_main}
\dot{x}=F(t,x)
\end{equation}
with $\dot{x}:=\frac{dx}{dt}$. Following closely the setting
in~\cite[Chapter~9]{ANM-LH-DL:08}, a continuous function
$\phi:[t_0,t_0+c)\to \mcX$, $c>0$, is a solution of~\eqref{eq:dyn_main} if
  it is differentiable with respect to $t$ for $t\in[t_0,t_0+c)$ and if
    $\phi$ satisfies the equation $\dot{\phi}=F(t,\phi(t))$ for all
    $t\in[t_0,t_0+c)$. When the system~\eqref{eq:dyn_main} is associated
      the initial condition $x(t_0)=x_0$, we have an initial value problem
      or Cauchy problem. In this paper we assume that (i) $t\mapsto F(t,x)$ is continuous 
      on $t\in\R_{\geq0}$ for any $x\in\mcX$, and that (ii) for any $x_0\in\mcX$,
      there exists at least one solution $\phi(t,t_0,x_0)$ to the initial
      value problem with $x(t_0)=x_0=\phi(t_0,t_0,x_0)$ for all $t\geq t_0$, 
      $t_0\in\R_{\geq0}$.

We say that a set $\mathcal{U}$ is (positively) invariant for the
system~\eqref{eq:dyn_main} if $\phi(t',t_0,x_0)\in\mathcal{U}$ at some time
$t'\geq t_0$ implies $\phi(t,t_0,x_0)\in\mathcal{U}$ for any $t\geq t'$.

The dynamical system~\eqref{eq:dyn_main} is time-invariant whenever the
vector field $F$ is time-invariant, i.e., $F$ does not explicitly depend on
$t$.  If the system~\eqref{eq:dyn_main} is time-invariant, it has an
equilibrium point $x^*$ if $F(x^*)=0$.

The system~\eqref{eq:dyn_main} has exponential 
  incremental stability or is \emph{contracting} with respect to norm $\norm{\cdot}$ if, for any $x_0,y_0\in\mcX$, the  
  trajectories $\phi(t,t_0,x_0)$
  and $\phi(t,t_0,y_0)$ for any $t\geq t_0$ satisfy
  $\norm{\phi(t,t_0,x_0)-\phi(t,t_0,y_0)}\leq e^{-c(t-t_0)}\norm{x_0-y_0}$, for some $c>0$. 
 %\pc{This constant $c$ is known as the contraction rate.}
%  A system is \emph{contractive} with
%respect to norm $\norm{\cdot}$ when $c>0$, with $c$ known as the \emph{contraction rate}. 
%%If the system is contractive and time-invariant, then there must exist a unique
%%equilibrium point which is %and so this equilibrium point is 
%%globally exponentially stable. 
In the Euclidean case, a central tool for studying contractivity is the matrix
measure~\cite{ZA-EDS:14b}, i.e., the operator measure taking matrices as arguments.

%Assume the Jacobian of the system~\eqref{eq:system}, i.e., $Df(t,x)$, satisfies:
%$\mu(Df(t,x))\leq -c$ for any $(t,x)\in
%\R_{\geq 0}\times\R^n$, with $\mu$ being the induced matrix measure from norm $\norm{\cdot}$ and $c>0$ being some constant; 
%then, %a very known result of contraction theory is that 
%the system
%has contraction rate $c$ with respect to
%$\norm{\cdot}$~\cite{DCL:49,AP-AP-NVDW-HN:04}. 
%
%\pc{The concept of partial and semi-contraction have not been formalized before on Banach spaces, and thus we propose the following formalization.}\pcmargin{Could this be a contribution?}
%
%{\color{red}Now, consider 
%$V\in\R^{k\times{n}}$ being a full-row rank matrix with orthonormal
%rows, then we say that the system is \emph{semi-contractive} with respect to
%norm $\norm{\cdot}$ if there exists
%$c>0$ such that, for any $x_0,y_0\in\R^n$ and $t_0\in \real_{\geq 0}$, 
%it satisfies 
%$\norm{V\phi(t,t_0,x_0)-V\phi(t,t_0,y_0)}\leq \norm{Vx_0-Vy_0}e^{-c(t-t_0)}$. 
%If we assume that the
%system~\eqref{eq:system} has a flow-invariant linear subspace
%$\mathcal{M}=\setdef{x\in\R^n}{Vx=\vect{0}_k}$, 
%then the system is \emph{partially contractive} with respect to norm $\norm{\cdot}$ if there exists
%$c>0$ such that, for any $x_0\in\R^n$ and $t_0\in \real_{\geq 0}$, 
%it satisfies $\norm{V\phi(t,t_0,x_0)}\leq \norm{Vx_0}e^{-c(t-t_0)}$, i.e.,
%there is exponential convergence to $\mathcal{M}$~\cite{QCP-JJS:07}. }

\section{Contraction on Banach and Hilbert spaces}
\label{sec:contr_main}

The following Lemma~\ref{thm:coppel} was proved
in~\cite[Lemma~5.4.1]{GEL-VL:72},\footnote{The
  result~\cite[Lemma~5.4.1]{GEL-VL:72} does not prove the left inequality
  in equation~\eqref{mu_1}, but this follows immediately from the same
  proof.} and the next Theorem~\ref{theorem_princ2} is an application
of~\cite[Lemma~5.4.2]{GEL-VL:72}. These are the only two existing results
from the scarce literature on contraction on Banach spaces that we use.

\begin{lemma}[Coppel's inequality for Banach spaces {\cite[Lemma~5.4.1]{GEL-VL:72}}]
  \label{thm:coppel}
    Consider the linear time-varying dynamical system
    \begin{align*}
      \dot{x}(t) = A(t) x(t)
    \end{align*}
    on the Banach space $(\mcX,\norm{\cdot})$, with $A(t)\in\BX$ and $t\mapsto A(t)$ being continuous for every $t\in\R_{\geq 0}$. Suppose that $\phi(t,t_0,x_0)$ is a solution of the Cauchy problem, then 
    \begin{equation}
    \label{mu_1}
      \norm{x_0}\exp\left(\int_{t_0}^{t} -\mu(-A(\tau))d\tau\right)\le
      \|\phi(t,t_0,x_0)\| \le \norm{x_0}\exp\left(\int_{t_0}^{t}\mu(A(\tau))d\tau\right).
      \end{equation}
    \end{lemma}

\begin{theorem}[Contraction with operator measures on Banach Spaces.  {\cite[Lemma~5.4.2]{GEL-VL:72}}]
\label{theorem_princ2}
Consider the dynamical system~\eqref{eq:dyn_main} on the Banach space
$(\mcX,\norm{\cdot})$ with $F(t,\cdot)$ continuously Fr\'{e}chet
differentiable for each $t$ and such that $\norm{DF(t,u)}_{\op}\leq a(t)$ for any
$u\in\mcX$, and for some continuous function $a(t)\geq 0$, $t\geq 0$. Assume that $\mu(DF(t,x))\leq -c(t)$ for any $x\in\mcX$, and some continuous function $c(t)>0$, $t\geq 0$.
%for every $(t,x)\in\R_{\geq
%  0}\times\mcX$. 
Then the system~\eqref{eq:dyn_main} is \emph{contractive},
i.e.,
\begin{equation}
\label{eq:ctr-1}
\norm{\phi(t,t_0,x_0)-\phi(t,t_0,x_0')}\leq e^{-\int_{t_0}^tc(s)ds}\|x_0-x_0'\|
\end{equation}
for all $t\geq t_0$ and any $x_0,x_0'\in\mcX$.
\end{theorem}
The proof can be found in the Appendix. The function $c(t)>0$, $t\geq 0$, in Theorem~\ref{theorem_princ2} is the \emph{contraction rate}. When the contraction rate is time-invariant, then a contractive system has incremental exponential stability.
Beginning now, all %of 
the following results presented in this paper are novel.
%
%\begin{remark}[Weighted norms]
%%The notion of \emph{weighted norms} can also be generalized and define norms in more general Banach spaces as follows: 
%Let $\mathcal{T}$ %:\mcX\to\mcX$ 
%be an invertible element of $\BX$, then, for
%any $x\in\mcX$, the function $x\mapsto\norm{\mathcal{T}x}$ defines an equivalent 
%norm to $\norm{\cdot}$. 
%Moreover, for this weighted norm, the operator measure of
%$A\in\BX$ is easily shown to be
%$\mu(\mathcal{T}A\mathcal{T}^{-1})$. 
%  In practice, this observation can be used in
%  finding a suitable norm with
%  respect to which a system is contractive. 
%See~\cite{PCV-SJ-FB:19r} for a recent application in the Euclidean space.
%\end{remark}
%
We present additional properties for contractive systems when the vector field is time-invariant.

\begin{theorem}[Properties of time-invariant systems]%[Time-invariant contractive systems]
\label{theorem_princ3}
Consider the dynamical system~\eqref{eq:dyn_main} on the Banach space
$(\mcX,\norm{\cdot})$ with $F$ time-invariant. Pick  $c>0$.
%If the system is contractive with contraction rate $c$, then
  \begin{enumerate}
\item\label{p2:contractive}If the system is contractive
with %\pc{constant} 
contraction rate $c$, then there exists a unique globally exponentially stable
  equilibrium point $x^*$ such that
\begin{equation*}
  \norm{\phi(t,t_0,x_0)-x^*}\leq e^{-c(t-t_0)}\|x_0-x^*\|,
\end{equation*}
for all $t\geq t_0$ and any $x_0\in\mcX$.%; and 
  \item\label{p3:contractive} 
  If $F$ is continuously Fr\'{e}chet differentiable and $\mu(\jac F(x))\le -c$ for every  $x\in\mcX$, then $\|F(\phi(t,t_0,x_0))\|\le e^{-c(t-t_0)}\|F(x_0)\|$, %for every
for all $t\geq t_0$ and any $x_0\in\mcX$.
  \end{enumerate}
\end{theorem}
\begin{proof}
We prove statement~\ref{p2:contractive}. Recalling that the system is contractive, we have $\norm{\phi(t,t_0,x_0)-\phi(t,t_0,x_0')}\leq e^{-c(t-t_0)}\norm{x_0-x_0'}$ for any $t\geq t_0$, $x_0,x_0'\in\mcX$. Fix any $t> t_0$. Since $e^{-c(t-t_0)}<1$, we can use the Banach fixed point theorem to conclude that there exists a unique fixed point $x^*$ such that $\phi(t,t_0,x^*)=x^*$, which implies that $x^*$ is either an equilibrium point or is a point which is revisited by the trajectory at time $t$. 
By contradiction, if we assume the latter, then any point $y^*$ at time $t_0\leq t'\leq t$ will be revisited at time $t+(t'-t_0)$ (since there is uniqueness of solutions from~\eqref{eq:ctr-1} and $F$ is time invariant) and thus $y^*$ is also
a fixed point of $\phi(t,t_0,\cdot)$, which violates the uniqueness of $x^*$ as a fixed point. Then, $x^*$ must be the unique equilibrium of $F$. We just proved statement~\ref{p2:contractive}.

Finally, to prove statement~\ref{p3:contractive}, observe that using the chain rule on Banach spaces~\cite[Theorem 2.4.3]{RA-JEM-TSR:88}, $\frac{d}{d t}   F(\phi(t,0,x_0) )
    = DF(\phi(t,t_0,x_0) ) 
    \frac{d}{d t}   \phi(t,t_0,x_0)=DF(\phi(t,t_0,x_0)) F(\phi(t,t_0,x_0))$, i.e., $F(\phi(t,t_0,x_0))\in\mcX$ satisfies a linear time-varying differential equation on Banach spaces.
 Now, using
  Lemma~\ref{thm:coppel},
\begin{equation}
  \label{eq:bound-Coppel}
  \left\|F(\phi(t,t_0,x_0))\right\|\leq 
   \left\|F(x_0)\right\|e^{\int_{t_0}^t\mu\big( \jac{F}(\phi(\tau,t_0,x_0)) \big) d\tau}
      \leq e^{-c(t-t_0)} \left\|F(x_0)\right\|,
\end{equation}
where we used $\mu(\jac F(x))\le -c$, for every  $x\in\mcX$. 
\end{proof}

Assume now that $\mcX$ is also a Hilbert space (over the field of real
numbers) equipped with some inner product
$\langle\cdot,\cdot\rangle$. Then, a weaker and simpler sufficient
condition for contractivity than the one in Theorem~\ref{theorem_princ2}
can be obtained.
% if the dynamical system~\eqref{eq:dyn_main} is time-invariant.

\begin{theorem}[Contraction on Hilbert spaces]%Integral contractivity condition]
\label{integral-contractivity}
Consider the dynamical system~\eqref{eq:dyn_main} on the Hilbert space
$(\mcX,\inpr{\cdot}{\cdot})$.% with $F$ time-invariant.  
\begin{enumerate}
\item\label{it:1-ic}If the following
\emph{integral contractivity condition} holds
\begin{equation}
  \label{eq:int-cont-cond}
  \langle x-y,F(t, x)-F(t, y)\rangle\leq -c(t) \norm{x-y}^2
\end{equation}
for some continuous function $c(t)>0$, $t\geq 0$, and any $x,y\in\mcX$, then system~\eqref{eq:dyn_main}  is
contractive.
\item\label{it:2-ic} If $F$ is continuously Fr\'{e}chet differentiable,
and $\mu(DF(t,x))\leq -c(t)$ for any $x\in\mcX$, and some continuous function $c(t)>0$, $t\geq 0$, 
then condition~\eqref{eq:int-cont-cond} holds.
\end{enumerate}
\end{theorem}
\begin{proof}
Consider~\eqref{eq:int-cont-cond} and define $e:=x-y$. Then,
$\dot{e}=F(t,x)-F(t,y)$ and we obtain 
%$\langle e,\dot{e}\rangle\leq
%-c\norm{e}^2\Rightarrow \frac{d^+\norm{e}^2}{dt}=\frac{d^+\langle
%  e,e\rangle}{dt}=(\langle e,\dot{e}\rangle+\langle \dot{e},
%e\rangle)\leq -2c\norm{e}^2$, 
$\langle e,\dot{e}\rangle\leq
-c\norm{e}^2\Rightarrow \frac{d\norm{e}^2}{dt}=\frac{d\langle
  e,e\rangle}{dt}=(\langle e,\dot{e}\rangle+\langle \dot{e},
e\rangle)\leq -2c\norm{e}^2$,
%
%\sabertodo{what does $d^+$ mean in
%  this context? Dinni Derivative? Can we use Gr\"{o}nwall's Lemma for
%  this derivative?} 
  where we used the fact that the inner product is a bilinear function. Solving this differential inequality using the 
 Gr\"{o}nwall's Lemma, 
we obtain $\norm{e(t)}\leq e^{-\int_{t_0}^tc(s)ds}\norm{e(t_0)}$ for any $t\geq t_0$, establishing that the system is contractive and proving statement~\ref{it:1-ic}.

Now, we prove statement~\ref{it:2-ic} of the theorem. First, let $A\in\BX$, then
\begin{equation}
\label{eq:lower_bound}
\begin{split}
\mu(A)&=\lim_{h\to 0^+}\frac{\sup_{\substack{x,y\neq 0\\x,y\in \mcX}}\frac{|\langle x,(I+hA)y\rangle|}{\norm{x}\norm{y}}-1}{h}\\
&\geq \lim_{h\to 0^+}\frac{\frac{\langle x,(I+hA)x\rangle}{\langle x,x\rangle}-1}{h}=\frac{\langle x,Ax\rangle}{\langle x,x\rangle}
%&=\lim_{h\to 0^+}\frac{\langle x,x\rangle}{h\langle x,x\rangle}+\frac{\langle x,Ax\rangle}{\langle x,x\rangle}-\frac{1}{h}\\
\end{split}
\end{equation}
for any $x\in\mcX$ and $x\neq 0$; the first equality follows from~\cite[p.~187]{JM:14}. However, note that $\mu(A)\langle x,x\rangle\geq\langle x,Ax\rangle$ does hold for any $x\in\mcX$. Now, consider $F$ to be continuously Fr\'{e}chet differentiable, and consider any $x,y\in\mcX$. From the fundamental theorem of calculus for Fr\'{e}chet derivatives~\cite[Proposition~2.4.7]{RA-JEM-TSR:88}, % gives us
\begin{equation}
%\begin{aligned}
\label{oin1}
F(t,x)-F(t,y)=\left(\int_0^1 DF(t,s_{\lambda}(x,y))d\lambda\right)(x-y)
%\end{aligned}
\end{equation}
for a fixed $t$, with $s_{\lambda}(x,y):=x+\lambda(x-y)$, and where the integral is the Riemann integral on Banach
spaces~\cite[Chapter~7]{KL:85}\cite[Section~1.3]{GEL-VL:72} ($B(\mcX)$
is a Banach space with the operator norm). Then, 
$\langle x-y,F(t,x)-F(t,y)\rangle=\langle x-y,\int_0^1 DF(s_{\lambda}(x,y))d\lambda (x-y)\rangle$, and using~\eqref{eq:lower_bound},% we obtain 
\begin{align*}
\langle x-y,F(t,x)-F(t,y)\rangle&\leq\mu\left(\int_0^1 DF(t,s_{\lambda}(x,y))d\lambda\right)\norm{x-y}^2\\
&\leq\int_0^1\mu\left(DF(t,s_{\lambda}(x,y)\right)d\lambda\norm{x-y}^2\leq -c(t)\norm{x-y}^2.
\end{align*}
%\pcmargin{Before, it simply said "where the second inequality follows from the sub-additive property of the operator measure~\cite[Problem~5.4.1]{GEL-VL:72} and the fact that the Riemann integral is the limit of a sum of elements from the appropriate Banach space". I think it needed to be proved or justified in a more rigorous way.}
We now justify the second inequality above. Let $S_n$ be the $n$th partial sum of a Riemann integral $\mathcal{I}$, and set $q_n(h)=\frac{\norm{I+hS_n}_{\op}-1}{h}$. Observe that (i) $\lim_{h\to0^+}q_n(h)=\mu(S_n)$ for each $n$; (ii) $\lim_{n\to\infty}q_n(h)=\frac{\norm{I+h\mathcal{I}}_{\op}-1}{h}$ uniformly over $h$. Then, the Moore-Osgood Theorem implies %$\lim_{n\to\infty}\lim_{h\to 0^+}q_n(h)=\lim_{h\to 0^+}\lim_{n\to\infty}q_n(h)$ and so
$\lim_{n\to\infty}\mu(S_n)=\mu(\mathcal{I})$. This and the sub-additive property of operator measures~\cite[Problem~5.4.1]{GEL-VL:72} prove the second inequality above.
%This is the sought condition~\eqref{eq:int-cont-cond}.
\end{proof}

\begin{remark}[About contraction on Hilbert spaces]
%On the integral contractivity
%    condition]
    \mbox{}\null\hspace{-1ex}\null
  \begin{enumerate}
\item The integral contractivity condition does not require the vector
  field $F$ to be Fr\'{e}chet differentiable.
\item When $F$ is continuously Fr\'{e}chet differentiable, $DF(t,\cdot)$ %the Jacobian 
  is no longer required to be %uniformly 
  bounded as in
  Theorem~\ref{theorem_princ2} (which follows
  from~\cite[Lemma~5.4.2]{GEL-VL:72}). Then, %for time-invariant systems, 
  statement~\ref{it:2-ic} of Theorem~\ref{integral-contractivity}
  provides a more relaxed condition for contraction using operator measures when the dynamical system
  is on Hilbert spaces.
\item The integral contractivity condition generalizes a known sufficient
  condition of contractivity (e.g.,~\cite[Lemma 2.1]{PCV-SJ-FB:19r}) that
  has been established in the Euclidean space and is related to the
  so-called QUAD condition for dynamical
  systems~\cite{PD-MdB-GR:11,AP-AP-NVDW-HN:04}.
\end{enumerate}
\end{remark}

\begin{remark}[Uniqueness of solutions]
For any system satisfying the assumptions of Theorem~\ref{theorem_princ2} or Theorem~\ref{integral-contractivity}, the existence of a solution implies its uniqueness.%to the initial value problem
\end{remark}

\section{Semi- and partial contraction on Hilbert spaces}
\label{sec:part-contr-HS}

In this section, let $(\mcX, \inprX{\cdot}{\cdot})$ and
$(\Y,\inprY{\cdot}{\cdot})$ be Hilbert spaces and let $\T:\mcX\to\Y$ be
linear, surjective and bounded.\footnote{In this case, the operator norm of
  $\T$ is $\norm{\T}_{\op}=\sup_{\substack{x\in\mcX\\x\neq 0}}\frac{\norm{\T
      x}_\Y}{\norm{x}_\mcX}$.}  A classic example is $\mcX=\R^n$, $\Y=\R^m$
with $m\leq n$, and $\T\in\R^{m\times n}$ being a full rank matrix.

Define the bilinear function
$\semin{\cdot}{\cdot}_{\T}:\mcX\times\mcX\to\R$ by
$\semin{x_1}{x_2}_{\T}=\inprY{\T x_1}{\T x_2}$, and define the seminorm
$\norm{x_1}_{\T}:=\sqrt{\semin{x_1}{x_1}_{\T}}$.  Let $\T^{\dagger}$ be the
Moore-Penrose (generalized) inverse of $\T$, which is a well-defined
operator since $\T$ is surjective (and trivially has closed
range)~\cite[Corollary~11.1.1]{GW-YW-SQ:18}.

\begin{definition}[Partial and semi-contraction]
The system~\eqref{eq:dyn_main} is 
\begin{enumerate}
\item \emph{partially contractive} with respect to~$\norm{\cdot}_{\T}$ if
  there exists a continuous function $c(t)>0$, $t\geq 0$, such that, for any $x_0\in\mcX$ and $t\geq t_0$,
  \begin{equation}\label{eq:part-contr}
    \norm{\phi(t,t_0,x_0)}_{\T}\leq e^{-\int_{t_0}^tc(s)ds} \norm{x_0}_{\T};
  \end{equation}
  
\item \emph{semi-contractive} with respect to~$\norm{\cdot}_{\T}$ if there
  exists a continuous function $c(t)>0$, $t\geq 0$, % $c>0$ 
  such that, for any $x_0,y_0\in\mcX$ and $t\geq t_0$,
  \begin{equation}\label{eq:semi-contr}
  \norm{\phi(t,t_0,x_0)-\phi(t,t_0,y_0)}_{\T}\leq e^{-\int_{t_0}^tc(s)ds}
  \norm{x_0-y_0}_{\T}.
  \end{equation}
\end{enumerate}
\end{definition}

We remark that the concept of partial and semi-contraction have not been formalized before on Hilbert  spaces. Indeed, in the Euclidean space (with the usual inner product) with time-invariant contraction rates, our formalization becomes the classic cases studied in~\cite{SJ-PCV-FB:19q} and~\cite{QCP-JJS:07} respectively, where $\T$ becomes an $n\times{m}$, $n<m$, full-row rank matrix.

We introduce the following useful concepts. 

\begin{definition}[$\T$-seminorms and $\T$-operator semi-measures]
\label{def:restricted_norm_op}
Consider a linear, surjective, bounded operator $\T:\mcX\to\Y$ and let
$A\in\BX$. The associated \emph{$\T$-seminorm} of $A$ as
$$
\norm{A}_{\T,\op}=\sup_{\substack{x\in\ker(\T)^\perp\\x\in\mcX, x\neq 0}}\frac{\norm{Ax}_{\T}}{\norm{x}_{\T}}
$$
and the \emph{$\T$-operator semi-measure} of $A$ as 
$$
\mu_{\T} (A)=\lim_{h\to 0^+}\frac{\norm{I+hA}_{\T,\op}-1}{h}.
$$
\end{definition}

The definition of $\T$-operator semi-measure is well-posed, since the existence of directional derivatives follows from the subaddivity property of the seminorm $\norm{\cdot}_{\T,\op}$ (with the argument in $\BX$) as shown in~\cite[Example~7.7.]{KL:85}.
%, and one can easily follow the steps in~\cite[Example~7.7.]{KL:85} to show the existence of directional derivatives. 

\begin{theorem}[Partial contraction on Hilbert spaces]
  \label{theorem_part_conc}
  Let $(\mcX, \inprX{\cdot}{\cdot})$ and $(\Y,\inprY{\cdot}{\cdot})$ be
  Hilbert spaces and $\T:\mcX\to\Y$ be linear, surjective and bounded.
  Consider the dynamical system~\eqref{eq:dyn_main} on
  $(\mcX,\inprX{\cdot}{\cdot})$ with $F(t,\cdot)$ continuously Fr\'{e}chet
  differentiable for each $t$. 
%   and such that $\norm{\T
%    DF(t,u)\T^\dagger}\leq a$ for any $u\in\mcX$, $t\geq 0$, and some
%  constant $a>0$.  
  Assume that
  \begin{enumerate}
  \item\label{ass:1} there exists a continuous function $c(t)>0$ such that $\mu_{\T} (DF(t,x))=\mu(\T DF(t,x)
    \T^{\dagger})\leq -c(t)$ for every $(t,x)\in\R_{\geq 0}\times\mcX$ (with
    the operator measure $\mu$ associated to~$\norm{\cdot}_Y$),
  \item\label{ass:2} the subspace $\ker(\T)$ is positively invariant.
  \end{enumerate}
  Then the system~\eqref{eq:dyn_main} is partially contractive with respect to~$\norm{\cdot}_{\T}$.
\end{theorem}

\begin{proof}
We first observe that
\begin{equation}
\label{eq:aux_res}
\begin{aligned}
\norm{A}_{\T,\op}&=\sup_{\substack{x\in\ker(\T)^\perp\\x\in\mcX,x\neq 0}}\frac{\norm{\T A\T^{\dagger}\T x}_{Y}}{\norm{\T x}_{Y}}\\
&=\sup_{\substack{y\neq 0\\y\in\Y}}\frac{\norm{\T A\T^{\dagger}y}_{Y}}{\norm{y}_{Y}}=\norm{\T A\T^{\dagger}}_{Y,\op}
\end{aligned}
\end{equation}
where the first equality follows from the fact that $\T^{\dagger}\T$ is a
projection operator on
$\ker(\T)^\perp$~\cite[Theorem~3.5.8]{TH-RW:15}. Then, using the fact that
$\T\T^{\dagger}=I$, which follows from $\T$ being
surjective~\cite[Definition~11.1.3]{GW-YW-SQ:18}, it follows that
\begin{equation}
\label{eq:aux_res1}
\begin{aligned}
\mu_{\T} (A)&=\lim_{h\to 0^+}\frac{\norm{\T(I+hA)\T^{\dagger}}_{Y,\op}-1}{h}\\
&=\lim_{h\to 0^+}\frac{\norm{I+h\T A\T^{\dagger}}_{Y,\op}-1}{h}=\mu(\T A\T^{\dagger}).
\end{aligned}
\end{equation}
Now, set 
$y=\T x$, with $x$ being the state of the system, and 
%since $x$ is (Fr\'{e}chet) differentiable with 
%respect to time, 
by the chain rule, $y$ is differentiable with respect to time and $\dot{y}=\T\dot{x}=\T F(t,x)$. Now, since $\T$ is a bounded linear operator, % (i.e., a bounded linear map), 
$\ker(\T)$ is a closed linear subspace of $\mcX$, and so, we have the
following decomposition
$\mcX=\ker(\T)\oplus\ker(\T)^\perp$~\cite[Theorem~1,
Section~3.4]{DGL:69}. Set $U:=I-\T^\dagger\T$. Then, for any
trajectory $t\mapsto x(t)$, we have $x(t)=\T^\dagger\T x(t)+Ux(t)=\T^\dagger y(t)+Ux(t)$, with $\T^\dagger y(t)\in \ker(\T)^\perp$ and $Ux(t)\in\ker(\T)$. Then,
\begin{equation}
\label{eq:auxy}
\dot{y}=\T F(t,\T^\dagger y+Ux(t))
\end{equation}
is a time-varying dynamical system on the Hilbert space $(\Y,\inprY{\cdot}{\cdot})$, and so the Fr\'{e}chet derivative, using the chain rule,
%\cite[Theorem~3.3.4.]{RSH:82}, 
%~\cite[Theorem 2.4.3]{RA-JEM-TSR:88}) 
of the right-hand side of~\eqref{eq:auxy} %the previous expression
(with respect to $y$) is $\T DF(t,\T^\dagger y+Ux(t))\T^\dagger$.
%Set the associated dynamical
%system be $\dot{y}=H(t,y)$. Then, using the chain rule
%%\cite[Theorem~3.3.4.]{RSH:82}, 
%~\cite[Theorem 2.4.3]{RA-JEM-TSR:88} 
%we obtain that the Fr\'{e}chet derivative of the right hand side is (in terms of $x$) $\T DF(t,x)\T^{\dagger}$.
Then, from~\eqref{eq:aux_res1}, it easily follows from Theorem~\ref{integral-contractivity} that: 
if $\mu(\T DF(t,x)\T^{\dagger})\leq -c(t)$ as in the theorem statement and assumption~\ref{ass:1}, then the dynamical system~\eqref{eq:auxy} is contracting with respect to the norm $\norm{\cdot}_{Y}$. 
Now, we make the following observation: let us consider a solution such that initial condition $x_o\in\ker(\T)$ at time $t_0$ implies $\phi(t,t_0,x_o)\in\ker(\T)$ and so $\T\phi(t,t_0,x_o)=0$ for any $t\geq t_0$ (by assumption~\ref{ass:2}). Differentiating, we obtain $\T F(t,\phi(t,t_0,x_o))=0$, which let us conclude that if $u\in\ker(\T)$, then $\T F(t,u)=0$.
In conclusion, there are two solutions known for the
system~\eqref{eq:auxy}: $y=0$ (because if $y=0$, then $\dot{y}=\T
F(t,Ux)=0$ follows from $Ux\in\ker(\T)$ as we just showed) and
$t\mapsto y(t)=\T x(t)$, and these two solutions should exponentially converge to each other due to
contraction. Then, equation~\eqref{eq:part-contr} follows from $\norm{y}_{Y}=\norm{\T x}_{Y}=\norm{x}_{\T}$.
%Then, since $\norm{y}_{Y}=\norm{\T x}_{X}=\norm{x}_{\T}$,
%equation~\eqref{eq:part-contr} follows immediately.
\end{proof}

%\begin{remark}[Uniqueness of solutions]
%It can be shown that the conditions on the Jacobian of the system as in Theorem~ \ref{theorem_part_conc} are sufficient to ensure the uniqueness of solutions to the initial value problem~\cite[page~143]{GEL-VL:72}.
%\end{remark}

\begin{theorem}[Integral partial contractivity condition]
\label{partial-integral-contractivity-subs}
Consider the dynamical system~\eqref{eq:dyn_main} on the Hilbert space
$(\mcX,\inprX{\cdot}{\cdot})$ %with $F$ time-invariant, 
and the linear,
surjective, bounded operator $\T:\mcX\to\Y$. If the following
\emph{integral partial contractivity condition} holds
\begin{equation}
\label{eq:part-int-cont-cond-subs}
\semin{x}{F(t,x)}_{\T}\leq -c(t)\norm{x}^2_{\T}
\end{equation}
for some continuous function $c(t)>0$, $t\geq 0$, and any $x\in\mcX$, then
the system is partially contractive with respect to $\norm{\cdot}_{\T}$, and, as a consequence,
$\ker(\T)$ is a positively invariant subspace.%the subspace $\ker(\T)$. 
\end{theorem}
\begin{proof}
The proof is very similar to the first part of Theorem~\ref{partial-integral-contractivity} for proving its respective integral condition, and thus is omitted.
\end{proof}

We now introduce the counterpart of Theorem~\ref{integral-contractivity} for semi-contractive systems.

\begin{theorem}[Integral semi-contractivity condition]
\label{partial-integral-contractivity}
Consider the dynamical system~\eqref{eq:dyn_main} on the Hilbert space
$(\mcX,\inprX{\cdot}{\cdot})$ %with $F$ time-invariant, 
and the linear,
surjective, bounded operator $\T:\mcX\to\Y$.
\begin{enumerate}
\item\label{it:1-pic} If the following
\emph{integral semi-contractivity condition} holds
\begin{equation}
\label{eq:part-int-cont-cond}
\semin{x-y}{F(t,x)-F(t,y)}_{\T}\leq -c(t) \norm{x-y}^2_{\T}%,\text{ for some }c>0,\text{ and any }x,y\in\ker(\T)^{\perp}
\end{equation}
for some continuous function $c(t)>0$, $t\geq 0$, %$c>0$ 
and any $x,y\in\mcX$, then the
system~\eqref{eq:dyn_main} is semi-contractive with respect to $\norm{\cdot}_{\T}$.
  \item\label{it:2-pic}If $F$ is continuously Fr\'{e}chet differentiable, $\mu(\T
    DF(t,x)\T^{\dagger})\leq -c(t)$, for some continuous function $c(t)>0$, $t\geq 0$, and $\ker(\T)$ is an invariant
    subspace for $DF(t,x)$, for every $x\in\mcX$ and $t\geq 0$, then
    condition~\eqref{eq:part-int-cont-cond} holds.
  \end{enumerate}
\end{theorem}
\begin{proof}
First, consider the inequality~\eqref{eq:part-int-cont-cond} and define $e:=x-y$ and follow the same procedure as in the proof of Theorem~\ref{integral-contractivity} to show that 
$\norm{e(t)}_{\T}\leq e^{-\int_{t_0}^tc(s)ds}\norm{e(t_0)}_{\T}$ for any $t\geq t_0$, thus establishing the system is 
semi-contractive and proving statement~\ref{it:1-pic}.

Now, we prove statement~\ref{it:2-pic} of the theorem. Consider $F$ to be continuously Fr\'{e}chet differentiable, and consider any $x,y\in\mcX$. 
%Note that $\ker(\T)^{\perp}$ is a convex space, and so, 
Then, 
using the fundamental theorem of calculus for Fr\'{e}chet derivatives~\cite[Proposition~2.4.7]{RA-JEM-TSR:88}, we obtain, for a fixed $t$, $F(t,x)-F(t,y)=B(t,x,y)(x-y)$ with $B(t,x,y):=\int_0^1 DF(t,y+\lambda(x-y))d\lambda$. 

Then,
\begin{align}
  \inprY{\T(x-y)}{\T(F(t,x)-F(t,y))} \nonumber &=\inprY{\T(x-y)}{\T B(t,x,y)(I-\T^{\dagger}\T+\T^{\dagger}\T)(x-y)} \nonumber \\
&=\inprY{\T(x-y)}{\T B(t,x,y)\T^{\dagger}\T(x-y)} \nonumber \\
&\leq \mu(\T B(t,x,y)\T^{\dagger})\inprY{\T(x-y)}{\T(x-y)} \nonumber \\
  &=\mu(\T B(t,x,y)\T^{\dagger})\norm{x-y}_{\T}^2.   \label{eq:aux_pic}
\end{align}
We now justify the second equality in~\eqref{eq:aux_pic}. First, the invariance assumption
implies that $DF(t,u)v\in\ker(\T)$ for any $v\in\ker(\T)$ and $u\in\mcX$, and so: $\T B(t,x,y)v=\T\int_0^1 DF(t,y+\lambda(x-y))d\lambda v=\int_0^1 \T DF(t,y+\lambda(x-y))v d\lambda= 0$. Then, we use this to obtain the second equality: $(I-\T^\dagger\T)(x-y)\in\ker(\T)$, and so $B(t,x,y)(I-\T^{\dagger}\T)(x-y)\in\ker(\T)$  
and so $\T B(t,x,y)(I-\T^{\dagger}\T)(x-y)=0$.

Now, observe that 
\begin{equation*}
\mu(\T B(t,x,y)\T^{\dagger})=\mu(\int_0^1 \T DF(t,y+\lambda(x-y)\T^{\dagger}d\lambda)\leq \int_0^1\mu(\T DF(t,y+\lambda(x-y)\T^{\dagger})d\lambda\leq -c(t),
\end{equation*}
where the first inequality is justified in the same way as in the last part of the proof of Theorem~\ref{integral-contractivity}. Then, using this relationship in~\eqref{eq:aux_pic}, we obtain $\inprY{\T(x-y)}{\T(F(t,x)-F(t,y))}\leq -c(t)\norm{x-y}_{\T}^2$, which is condition~\eqref{eq:part-int-cont-cond}.
\end{proof}

\begin{remark}[About partial and semi-contraction] \mbox{}
\begin{enumerate}
\item The integral semi- and partial contractivity conditions do not
  require $F$ to be continuously Fr\'{e}chet differentiable.

\item If $\ker(\T)$ is positively invariant for the system, then the integral condition in
  Theorem~\ref{partial-integral-contractivity} implies partial
  contractivity.

\item For continuous differentiable vector fields on Euclidean spaces, the
  semi-contraction condition in
  Theorem~\ref{partial-integral-contractivity} was first introduced
  in~\cite{SJ-PCV-FB:19q}.
\end{enumerate}
\end{remark}

\section{Application to reaction-diffusion systems}
\label{sec:examples}
Reaction-diffusion PDEs have a long history of study due to their
importance in chemistry and biology~\cite{JDM:02}. Of particular interest
are conditions under which the system does not present the phenomenon of
pattern formation, which occurs from diffusion-driven
instabilities~\cite{MA:11}. Particular instances of these systems have been
studied using analysis related to contraction~\cite{ZA-YS-MA-EDS:14,ZA-ES:13,ZA-EDS:14b}.

Consider a bounded and convex domain $\Omega\subset\R^m$ with smooth
boundary $\partial\Omega$. For any function $h:\R^m\to\R^n$, define the
vector Laplacian operator $\nabla^2$ by $\nabla^2 h=(\nabla^2
h_1,\dots,\nabla^2 h_n)^\top$ and $(\nabla^2 h(x))_i = \sum_{j=1}^{m}
\frac{\partial^2 h_i(x)}{\partial x^2_j}$.
Let $\mcL^2(\Omega)$ be the space of squared-integrable functions $h:\Omega\subset\R^m\to\R^n$
with $\int_{\Omega}h_i^2dx<\infty$, $i\in\until{n}$, endowed with inner product $\inpr{u}{v}=\int_{\Omega}u^\top v dx$ for any $u,v\in\mcL^2(\Omega)$ %, and so $\int_\Omega u^\top u dx<+\infty$. 
and induced norm $\norm{u}=\sqrt{\inpr{u}{u}}$. %, $u\in\mcL^2(\Omega)$. 
It is known that $(\mcL^2(\Omega),\inpr{\cdot}{\cdot})$ is a Hilbert space.

%The set $\mcL^2(\Omega)$ with the inner product
%$\inpr{u}{v}=\int_{\Omega}u^\top(x) v(x)dx$ is known to be a Hilbert
%space. The induced norm is $\norm{u}=\sqrt{\inpr{u}{u}}$, $u\in
%\mcL^2(\Omega)$.

Given a continuously differentiable \emph{reaction function}
$f:\R^n\to\R^n$ and a nonnegative matrix of \emph{diffusion rates} $\Gamma
\in \real^{n\times n}$, the \emph{reaction-diffusion system with Neumann
  boundary conditions} is
\begin{equation}
  \label{ex_rdf}
  \begin{aligned}
    &\frac{\partial u}{\partial t}=f(u)+\Gamma\nabla^2u\\
    &\nabla u_i(t,x)\cdot \widehat{n}(x) = 0
    \quad\text{for all }x\in\partial\Omega,\,i\in\until{n},
\end{aligned}
\end{equation}
for $u\in \mcL^2(\Omega)$ and $(t,x)\in\R_{\geq 0}\times\Omega$, with $\widehat{n}(x)$ being the vector normal to $x\in\partial\Omega$. We refer
to~\cite{MA:11} and references therein for the system's well-posedness and
existence of classical solutions $u=u(x,t)$ such that $u(t,\cdot)$ is twice continuously differentiable for each fixed $t\in\R_{\geq 0}$, and that $t\mapsto
u(t,\cdot)$ is a twice continuously differentiable function on $\Omega$. We assume that
classical solutions exist.

A Neumann eigenvalue $\lambda\in\real$ for the Laplacian operator
$\nabla^2$ on $\Omega$ is defined by
\begin{equation*}
\begin{aligned}
&-\nabla^2u = \lambda u \\
&\nabla u_i(x)\cdot \widehat{n}(x) = 0\text{ for all }x\in\partial\Omega,\,i\in\until{n}. 
\end{aligned}
\end{equation*}
The set of Neumann eigenvalues of the Laplacian operator consists of
countably many nonnegative values with no finite accumulation
point~\cite[Section 7.1]{HA:06}: 
%\begin{align*}
$0 = \lambda_1 \le \lambda_2 \le \ldots\le  \lambda_k\le \ldots$
%\end{align*}
, i.e., $\lim_{k\to\infty}\lambda_k=\infty$. 
For our $\Omega$,  
% If $\Omega$ is simply-connected, then 
 the
eigenspace associated with the lowest eigenvalue $\lambda_1=0$ is
\begin{equation}
\label{eq:setS}
\begin{aligned}
  \mathcal{S} =\{h\in \mcL^2(\Omega) \;|\; h(x) = c \mbox{ for all } x\in\Omega\;\mbox{ and some constant vector }c\}. 
\end{aligned}
\end{equation}
The volume of $\Omega$ 
%a bounded domain $\Omega\subset \real^m$ 
is $|\Omega| =
\int_{\Omega} dx$ % = \mathrm{Vol}(\Omega)$ 
and the spatial average of $h\in
\mcL^2(\Omega)$ over $\Omega$ is $\overline{h} =
\frac{1}{|\Omega|}\int_{\Omega} h(x) dx\in \real^n$.  Define the operator
$\Pi_{\mathcal{S}}:\mcL^2(\Omega)\to \mathcal{S}^{\perp}$ by
  \begin{align}\label{eq:pi}
    \Pi_{\mathcal{S}}(u) = u - \overline{u}.
  \end{align}
  One can easily check that $\Pi_{\mathcal{S}}$ is the orthogonal
  projection onto $\mathcal{S}^{\perp}$. 
Since $\Pi_{\mathcal{S}}$ is surjective, i.e., $\im(\Pi_{\mathcal{S}}) =
        \mathcal{S}^{\perp}$, it follows from the definition 
of the Moore-Penrose pseudoinverse and its uniqueness 
that 
%the Moore-Penrose pseduoinverse of
%      $\Pi_{\mathcal{S}}$ exists 
%      and is given by
      $\Pi^{\dagger}_{\mathcal{S}}: \mathcal{S}^{\perp}\to
      \mcL^2(\Omega)$ is given by %and 
      $\Pi^{\dagger}_{\mathcal{S}}(u) = u$.

Given a matrix $A\in\R^{n\times{n}}$, the matrix measure associated to
  the standard Euclidean $2$-norm, $\mu_2(A)$, has the following 
  % is equal to the maximum eigenvalue of
  %$\frac{A+A^\top}{2}$~\cite{ZA-EDS:14b}.  A 
  property~\cite{ZA-EDS:14b}: 
  $\mu_{2}(A)\leq c$ if and only if $\frac{A+A^\top}{2}\preceq cI_n$, i.e.,
  $cI_n-\frac{A+A^\top}{2}$ is positive semi-definite.

\begin{theorem}[Partial and semi-contraction of reaction-diffusion systems]
\label{th:part-RD}
Consider the reaction-diffusion system~\eqref{ex_rdf} with the standard assumptions on $f$, $\Gamma$, and over a bounded and
convex set domain $\Omega\subset\R^m$. Suppose that there exists a positive
definite matrix $P\in\R^{n\times n}$ such that $\mu_2(P\Gamma)\geq 0$ and
$\mu_2(P(Df(x)-\lambda_2\Gamma))\leq -c$ for all $x\in\Omega$ and some
constant $c>0$. Define
$u\mapsto\norm{u}_{\Pi_{\mathcal{S}},P^{1/2}}:=\norm{\Pi_{\mathcal{S}}(P^{1/2}u)}$
with the set $\mathcal{S}$ as in~\eqref{eq:setS},
and let $\lambda_{\max}(P)$ be the largest eigenvalue of $P$.  Then,
\begin{enumerate}
  \item \label{it:ex-1}
system~\eqref{ex_rdf} is partially contractive with respect to
$\norm{\cdot}_{\Pi_{\mathcal{S}},P^{1/2}}$, that is, for every solution $u:\real_{\ge 0}\times \Omega\to \real^n$,
\begin{equation*}
 \norm{u(t,\cdot)}_{\Pi_{\mathcal{S}},P^{1/2}} \leq e^{-\frac{c}{\lambda_{\max}(P)}}\norm{u(0,\cdot)}_{\Pi_{\mathcal{S}},P^{1/2}},    
  \end{equation*}
\item \label{it:ex-2} $\ker(\Pi_{\mathcal{S}})=\mathcal{S}$ is an invariant subspace and all trajectories exponentially converge to it; and 
  \item \label{it:ex-3} if additionally $f(u)=Au$,
    $A\in\R^{n\times{n}}$, then~\eqref{ex_rdf} is semi-contractive with respect to
$\norm{\cdot}_{\Pi_{\mathcal{S}},P^{1/2}}$, that is, for every solution $u,v:\real_{\ge 0}\times \Omega\to \real^n$,
\begin{equation*}
 \norm{u(t,\cdot)-v(t,\cdot)}_{\Pi_{\mathcal{S}},P^{1/2}} \\\leq e^{-\frac{c}{\lambda_{\max}(P)}}\norm{u(0,\cdot)-v(0,\cdot)}_{\Pi_{\mathcal{S}},P^{1/2}}.
  \end{equation*}
\end{enumerate}
\end{theorem}

The proof can be found in the Appendix.

\begin{remark}
  Statements~\ref{it:ex-1} and~\ref{it:ex-2} of Theorem~\ref{th:part-RD}
  are essentially the same result as~\cite[Theorem~1]{MA:11}; however,
    these statements and statement~\ref{it:ex-3} are now consequences of a
    general contraction theory.
\end{remark}

In Figure~\ref{f:sim1} we simulate the following instance of~\eqref{ex_rdf}: $u=\begin{bmatrix}u_1\\u_2
\end{bmatrix}$, $\Gamma=\begin{bmatrix}
1&0\\0&2
\end{bmatrix}$, $f(u)=0.05\begin{bmatrix}
1/2&1/3\\1/3&1/4
\end{bmatrix}\begin{bmatrix}
u_1\\u_2
\end{bmatrix}$, $\Omega=[0,1]$. 
Set $\alpha(x):=-0.25+0.9x\sin(15x)$ and $\beta(x):=1.75x-0.75$, $x\in\Omega$. In the upper plot of Figure~\ref{f:sim1}, we plot the quantity $\norm{u(t,\cdot)-v(t,\cdot)}_{\Pi_{\mathcal{S}},I_2}$, $t\in[0,0.8]$, for signals $u$ and $v$ under the initial conditions $u_1(0,\cdot)=10\alpha$, $u_2(0,\cdot)=10\beta$, $v_1(0,\cdot)=5\beta$, and $v_2(0,\cdot)=2.5\alpha$. As predicted by the statement~\ref{it:ex-3} of Theorem~\ref{th:part-RD}, this quantity has en exponential decay due to semi-contraction. In the lower three plots, we plot $u_1(t,\cdot)$ for different values of $t$.

\begin{figure}[t]%[H]%[htp]
  \centering
  % \subfloat[$a=0.25$]{\includegraphics[width=0.33\linewidth]{plotTotalEvol2_025.pdf}} 
  \subfloat{\includegraphics[width=0.77\linewidth]{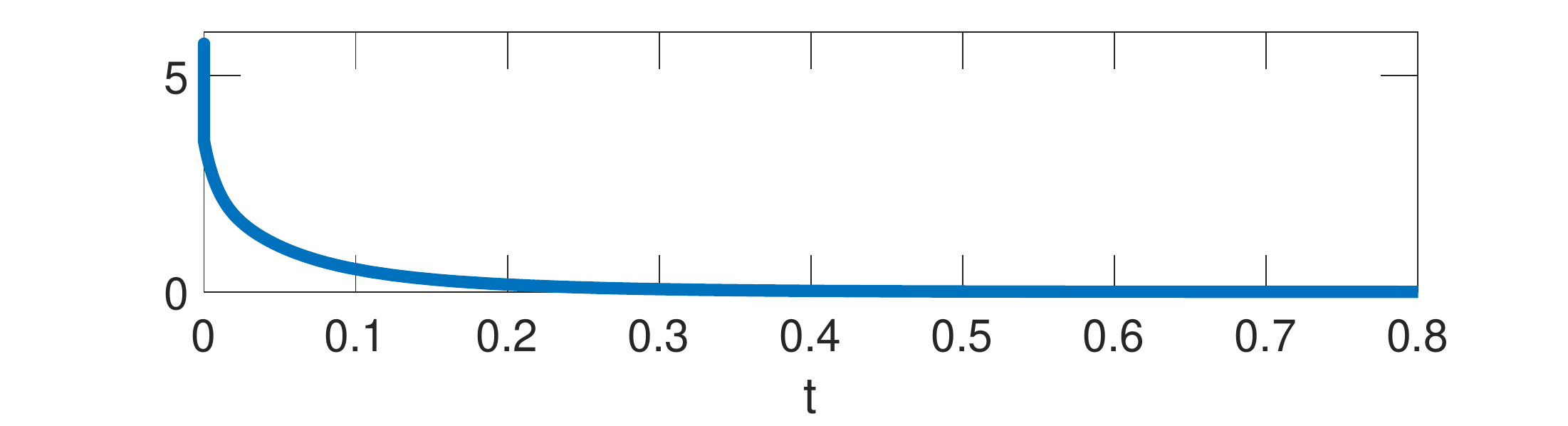}}\\
  \subfloat{\includegraphics[width=0.77\linewidth]{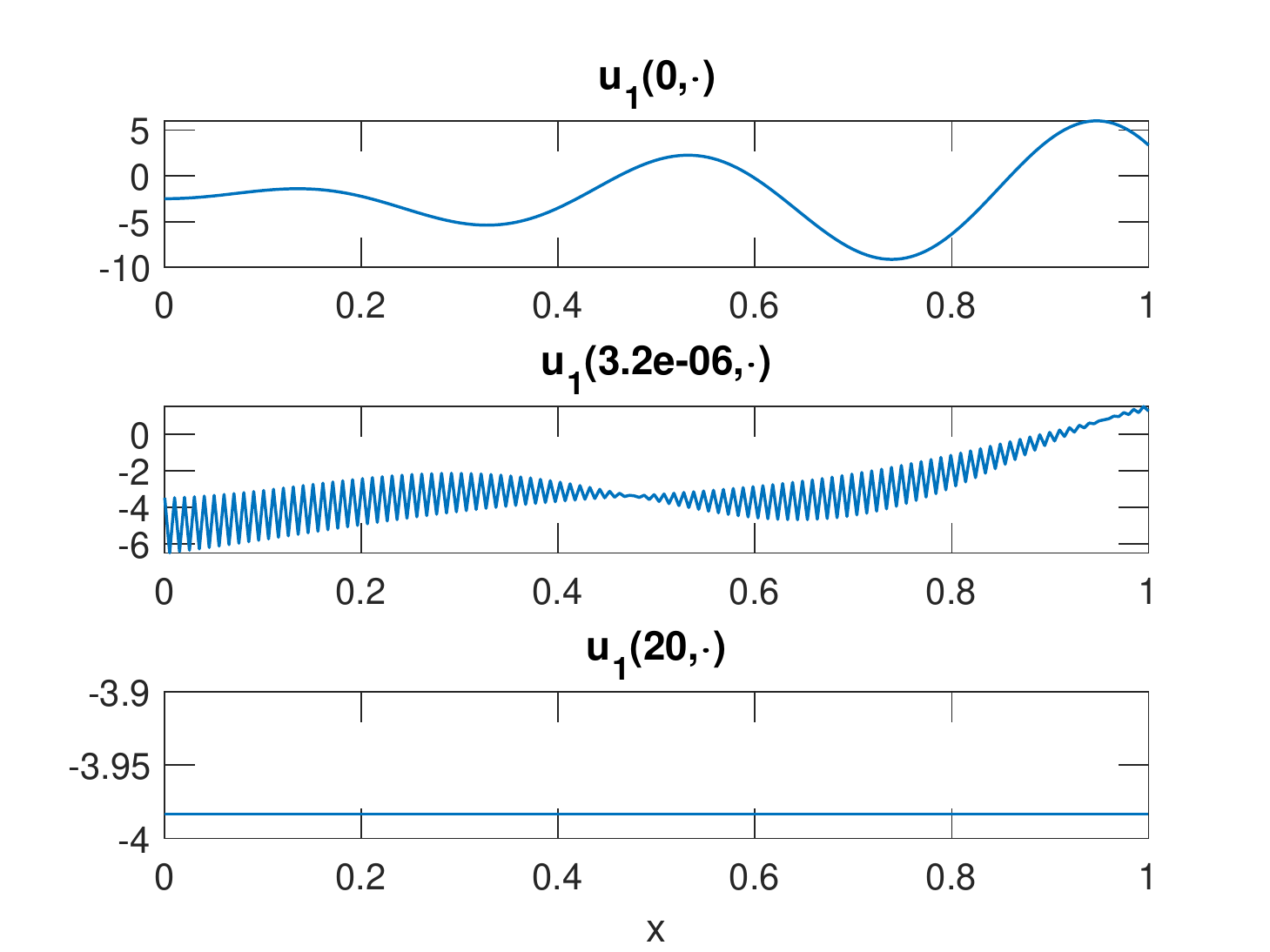}}
\caption{Numerical example for a reaction-diffusion system
%
%We simulated the following instance of~\eqref{ex_rdf}: $u=(u_1,u_2)$, $\Gamma$ with diagonal entries $1$ and $2$, $\Omega=[0,1]$. Set $x(t):=$ and $y(t):=$
}

  \label{f:sim1}
\end{figure}

\section{Conclusion}
\label{sec:concl}
This paper presents a general contraction theory for dynamical systems on
Hilbert spaces. We provide sufficient conditions for contraction,
semi-contraction and partial contraction based on operator measures or
operator semi-measures, and on the differentiability of the vector
field. Moreover, when the system is time-invariant, we present weaker conditions that do not require
differentiability. Finally, we present an example of reaction-diffusion
systems.

Our work brings the machinery of contraction theory, so far mainly applied
to ODEs, to other possible application domains related to a variety of
systems that can be expressed as dynamical systems on functional spaces.  
%We
%believe our work can potentially elicit a wide area of application.  As
%future work, we plan to apply the theory presented here to specific
%application problems, including the design of control strategies based on
%contraction theory on Banach spaces, inspired by recent developments for
%the classical Euclidean setting~\cite{IRM-JJES:17}.
%
%%(e.g., the work~\cite{IRM-JJES:17} includes examples of control design in the Euclidean setting). 
%%We also plan to study the phenomenon of weak contraction in Banach spaces, i.e., when the contraction rate is actually zero and all trajectories of a dynamical system have a non-expansive distance among themselves.
%%has the potential of impacting directly impacts applications on 
%
%%As future work, it would be interesting to study the phenomenon of weak contraction, i.e., properties related when the contraction rate of a system is equal to zero. The recent \pc{(or forthcoming)} work~ has provided an extensive characterization of properties related to weak contraction on Euclidean spaces.

\section*{Acknowledgment}
The authors thank Prof.\ Sam Coogan for insightful discussions about contraction theory. 

\bibliographystyle{plain}
\bibliography{alias,Main,FB}
\section{Apprendix}

\begin{proof}[\textbf{Proof of Theorem~\ref{theorem_princ2}}]
First, observe that, for any $u,h\in\mcX$, %$\norm{DF(t,u)h}\leq\norm{DF(t,u)}\norm{h}\leq a\norm{h}$.
$\norm{DF(t,u)h}\leq a(t)\norm{h}$. 
Moreover, observe that the differential equation $\dot{r}=a(t)r$ satisfies that if it has a solution $r(t)$ such that $r(t_0)=0$, then $r(t)= 0$ for any $t\geq t_0$. These two conditions satisfy the hypotheses of~\cite[Theorem~5.3.3]{GEL-VL:72}, and thus we can use~\cite[Lemma~5.4.2]{GEL-VL:72} from which equation~\eqref{eq:ctr-1} follows.
\end{proof}

\begin{proof}[\textbf{Proof of Theorem~\ref{th:part-RD}}]
Note that the reaction-diffusion system we are analyzing can be written as $\frac{\partial u}{\partial t} = F(u)$ where $F:\mcL^2(\Omega)\to\mcL^2(\Omega)$ is defined by $F(u):= f(u) + \Gamma\nabla^2u$. Let $\inprPRS{\cdot}{\cdot}:=\inpr{\Pi_{\mathcal{S}}(P^{1/2}\cdot)}{\Pi_{\mathcal{S}}(P^{1/2}\cdot)}$. 

We start by proving statement~\ref{it:ex-1}. Consider any $u\in \mcL^2(\Omega)$, and set $\tilde{u}=u-\bar{u}$, so that $\tilde{u}\in\mathcal{S}^\perp$. Then,
\begin{equation}
\label{eq:princ-pc}
\inprPRS{u}{F(u)}= \inprPRS{u}{f(u)}+\inprPRS{u}{\Gamma\nabla^2(u)}
\end{equation}

First, from the first term in the right-hand side of~\eqref{eq:princ-pc},
\begin{equation}
\label{eq:part-1}
\begin{aligned}
\inprPRS{u}{f(u)}&=\inpr{\Pi_{\mathcal{S}}(P^{1/2}(u))}{\Pi_{\mathcal{S}}(P^{1/2}(f(u)))}\\
%\inprPRS{\tilde{u}-\tilde{v}}{f(\tilde{u})-f(\tilde{v})}\\
%&= \int_{\Omega}(\tilde{u}-\tilde{v})^\top P(f(\tilde{u})-f(\tilde{v})-\int_{\Omega}(f(\tilde{u})-f(\tilde{v}))dx)dx\\
&= \int_{\Omega}\tilde{u}^\top P(f(u)-f(\bar{u}))dx\\
&=\int_0^1\int_{\Omega}\tilde{u}^\top PDf(s(\gamma))\tilde{u}dxd\gamma.
%&\leq \int_{\Omega}(\tilde{u}-\tilde{v})^\top f(u-v)dx\\
%%&=\int_{\Omega}(\tilde{u}-\tilde{v})^\top (f(u-v)-f(\bar{u}-\bar{v})\\
%%&\;+f(\bar{u}-\bar{v})-\left(\int_{\Omega}(f(u)-f(v))dx\right))dx\\
%&=\int_{\Omega}(\tilde{u}-\tilde{v})^\top (f(u-v)-f(\bar{u}-\bar{v}))dx
\end{aligned}
\end{equation}
where for the second equality we repeatedly used
$\int_{\Omega}\tilde{u}^\top a dx=0$ for any constant vector
$a\in\R^n$ in $\Omega$, and, since $\Omega$ is convex, the mean-value
theorem: $f(u)-f(\bar{u})=\int_0^1Df(s(\gamma))\tilde{u}d\gamma$
with $s(\gamma)=u+\gamma(\bar{u}-u)$ for the last inequality.
%
%\sabertodo{As can be seen in the course of the proof proper
%  introduction of $u$ at the begining of this section is essential,
%  where do $u$ and $u_i$ live?}
Now, from the second term in the right-hand side of~\eqref{eq:princ-pc},
\begin{equation}
\label{eq:tog1}
  \inprPRS{u}{\Gamma\nabla^2u}%=\inpr{\tilde{w}}{\Gamma\nabla^2\tilde{w}}
  =\int_{\Omega}\tilde{u}^\top P\Gamma\nabla^2\tilde{u}dx,
\end{equation}
where we used $\Pi_{\mathcal{S}}(\nabla^2u)=\nabla^2u-\frac{1}{|\Omega|}\int_{\Omega}\nabla^2udx=\nabla^2\tilde{u}$, since $\int_{\Omega}\nabla^2u_idx=\int_{\partial\Omega}\nabla u_i\cdot\widehat{n}\,d\,\partial\Omega=0$ (where the surface integral after the first equality follows from the divergence theorem, and the last equality follows from the boundary condition in~\eqref{ex_rdf}). 
Note that, for every $i\in \{1,\ldots,n\}$, we have $\nabla^2\tilde{u}_i(x) =
\nabla\cdot(\nabla \tilde{u}_i(x))$. Now, by the product rule, we have that for every $i\in \{1,\ldots,n\}$, 
$\nabla\cdot(\sum_{j=1}^{n} \tilde{u}_i (P\Gamma)
_{ij} \nabla \tilde{u}_j) = \sum_{j=1}^{n} \tilde{u}_i(P\Gamma)
_{ij} \nabla^2 \tilde{u}_j 
+ \sum_{j=1}^{n} (\nabla \tilde{u}_i)^{\top}(P\Gamma)_{ij}
  \nabla \tilde{u}_j$. 
Now, by the divergence theorem, we obtain for every $i\in \{1,\ldots,n\}$, 
$\int_{\Omega} \nabla\cdot(\sum_{j=1}^{n} \tilde{u}_i(P\Gamma)
_{ij} \nabla \tilde{u}_j) dx
=\int_{\partial \Omega} \big(\sum_{j=1}^{n} \tilde{u}_i(P\Gamma)_{ij}\nabla \tilde{u}_j\cdot\widehat{n} \big)dS  = 0$ 
where the last equality follows from the boundary condition in~\eqref{ex_rdf}. Then, from the identity $\tilde{u}^{\top}P\Gamma\nabla^2\tilde{u}= \sum_{i=1}^{n}\sum_{j=1}^{n} \tilde{u}_i(P\Gamma)_{ij} \nabla^2 \tilde{u}_j$ %and integrating %and some algebraic work, 
we get
$$  \int_{\Omega} \tilde{u}^{\top}P\Gamma
  \nabla^2\tilde{u} dx  =  -\int_{\Omega}\sum_{i=1}^{n}\sum_{j=1}^{n} (\nabla \tilde{u}_i)^{\top}(P\Gamma)_{ij}
  \nabla \tilde{u}_j dx.$$
Moreover, one can check that
\begin{align*}
  \sum_{i=1}^{n}\sum_{j=1}^{n} (\nabla \tilde{u}_i)^{\top} (P\Gamma)_{ij}
  \nabla \tilde{u}_j  = \sum_{k=1}^{m} (\frac{\partial \tilde{u}}{\partial
  x_k})^{\top}(P\Gamma)\frac{\partial \tilde{u}}{\partial
  x_k}.
\end{align*}
Since $\mu_{2}(P\Gamma) \ge 0$,  there exists a positive semi-definite matrix $Q\in \real^{n\times n}$ such that $Q^{\top}Q = \tfrac{1}{2}(P\Gamma+\Gamma^{\top}P^\top)$. This implies that
%\begin{multline*}
% \sum_{k=1}^{n} (\frac{\partial \tilde{w}}{\partial
%  x_k})^{\top}P\Gamma\frac{\partial \tilde{w}}{\partial
%  x_k}   = \sum_{k=1}^{n} (\frac{\partial \tilde{w}}{\partial
%  x_k})^{\top}Q^{\top}Q\frac{\partial \tilde{w}}{\partial
%  x_k}\\
%  = \sum_{k=1}^{n} (\frac{\partial Q\tilde{w}}{\partial
%  x_k})^{\top}\frac{\partial Q\tilde{w}}{\partial
%  x_k}
%  = \sum_{i=1}^{n} (\nabla ((Q\tilde{w})_i))^{\top}\nabla ((Q\tilde{w})_i).
%  \end{multline*}
%
$\sum_{k=1}^{m} (\frac{\partial \tilde{u}}{\partial
x_k})^{\top}P\Gamma\frac{\partial \tilde{u}}{\partial
x_k}   = \sum_{k=1}^{m} (\frac{\partial \tilde{u}}{\partial
x_k})^{\top}Q^{\top}Q\frac{\partial \tilde{u}}{\partial
x_k}= \sum_{k=1}^{m} (\frac{\partial Q\tilde{u}}{\partial
  x_k})^{\top}\frac{\partial Q\tilde{u}}{\partial
  x_k}
  = \sum_{i=1}^{n} (\nabla ((Q\tilde{u})_i))^{\top}\nabla ((Q\tilde{u})_i)$.

Combining all of these results, we finally obtain 
%  \begin{align*}
%    \int_{\Omega} \tilde{w}^{\top}P\Gamma
%   \nabla^2\tilde{w} dx  & = -\sum_{i=1}^{n}\|\nabla (Q\tilde{w})_i\|^2. 
%  \end{align*}
$    \int_{\Omega} \tilde{u}^{\top}P\Gamma
   \nabla^2\tilde{u} dx  = -\sum_{i=1}^{n}\|\nabla (Q\tilde{u})_i\|^2$. 
  Now, since $\int_{\Omega}Q\tilde{u}dx=Q\int_{\Omega}\tilde{u}dx=\vect{0}_n$, we can use the Poincar\'{e} inequality on simply
  connected domains~\cite[Section~1.3]{HA:06} %,ZA-YS-MA-EDS:14} 
%  (since the domain is convex, it is also simply connected),
(since our domain is convex),
   and obtain 
%  \begin{align*}
    $\|\nabla (Q\tilde{u})_i \|^2 \ge \lambda_2 \|(Q\tilde{u})_i\|^2$. 
 %   \end{align*}
  As a result, we get
  $\int_{\Omega} \tilde{u}^{\top}P\Gamma\nabla^2\tilde{u} dx \le
    -\lambda_2 \sum_{i=1}^{n}\|(Q\tilde{u})_i\|^2
    = \int_{\Omega} \tilde{u}^{\top} \left(-\lambda_2 P\Gamma\right) \tilde{u} dx$. 
  %\end{align*}  
%  
%  \begin{align*}
%    \int_{\Omega} \tilde{w}^{\top}P\Gamma\nabla^2\tilde{w} dx &\le
%    -\lambda_2 \sum_{i=1}^{n}\|(Q\tilde{w})_i\|^2\\
%    & = 
%    \int_{\Omega} \tilde{w}^{\top} \left(-\lambda_2 P\Gamma\right) \tilde{w} dx.
%  \end{align*}
  Replacing this result in~\eqref{eq:tog1}, and then replacing the resulting expression with the one in~\eqref{eq:part-1} back in~\eqref{eq:princ-pc}: 
  \begin{align*}
  \inprPRS{u}{F(u)}&= \int_0^1\int_{\Omega}\tilde{u}^\top P\left(Df(s(\lambda))-\lambda_2\Gamma\right)\tilde{u}dxd\gamma\\
&\leq-c\int_0^1\int_{\Omega}\tilde{u}^\top \tilde{u}dxd\gamma\leq-\frac{c}{\lambda_{\max}(P)}\norm{P^{1/2}\tilde{u}}\\
  &=-\frac{c}{\lambda_{\max}(P)}\norm{\tilde{u}}_{\Pi_{\mathbb{S}},P^{1/2}}=-\frac{c}{\lambda_{\max}(P)}\norm{u}_{\Pi_{\mathbb{S}},P^{1/2}}
  \end{align*}
where the inequality comes from the assumption
$\mu_2(P(Df(x)-\lambda_2\Gamma))\leq-c$ for any $x\in\Omega$. This expression
has the form of the integral partial contractivity condition. Although the set of classical solutions endowed with $\inpr{\cdot}{\cdot}$
  is not a Hilbert space, we follow the proof of 
  Theorem~\ref{partial-integral-contractivity-subs} using the Leibniz rule to differentiate the inner product and obtain partial contraction as in statement~\ref{it:ex-1}. Statement~\ref{it:ex-2} follows by noting that
$\norm{u}_{\Pi_{\mathcal{S}},P^{1/2}}=0
%\implies\Pi_{\mathcal{S}}(P^{1/2}u)=0
\implies P^{1/2}u\in\ker(\Pi_{\mathcal{S}})\implies P^{1/2}u\in\mathcal{S}\implies u\in\mathcal{S}$.
% and finish the proof.  %for statement~\ref{it:ex-1}.
Finally, statement~\ref{it:ex-3} is proved in a similar way to statement~\ref{it:ex-1}, using the difference of any two solutions as a new state variable.
%
% consider any two solutions $u,v$, and set $\tilde{w}=\tilde{u}-\tilde{v}$. Then, following a similar proof to statement~\ref{it:ex-1}, we can show that $\inprPRS{u-v}{F(u)-F(v)}= \inprPRS{u-v}{Au-Av}+\inprPRS{u-v}{\Gamma\nabla^2(u)-\Gamma\nabla^2(v)}=\int_{\Omega}\tilde{w}^\top PA\tilde{w}dx-\lambda_2\int_{\Omega}\tilde{w}^\top\Gamma\tilde{w}dx\leq-c\int_{\Omega}\tilde{w}^\top \tilde{w}dx\leq-\frac{c}{\lambda_{\max}(P)}\norm{\tilde{w}}_{\Pi_{\mathbb{S}},P^{1/2}}=-\frac{c}{\lambda_{\max}(P)}\norm{u-v}_{\Pi_{\mathbb{S}},P^{1/2}}$, which proves~\ref{it:ex-3} by using the integral semi-contractivity condition (Theorem~\ref{partial-integral-contractivity}).
% 
\end{proof}

\end{document}